\documentclass[11pt,a4paper]{article}
\usepackage{a4wide}
\usepackage{theorem}
\usepackage{amsmath,amssymb,calc}
\usepackage[english]{babel}
\usepackage{graphicx,array}
\usepackage[numbers]{natbib}
\usepackage{tikz}
\usepackage{color} 
\usepackage{enumerate}
\usepackage[toc,page]{appendix}
\usepackage[charter]{mathdesign}
\usepackage{algorithm}
\usepackage{algpseudocode}
\usepackage{tikz}
\usepackage{verbatim}

\graphicspath{{./},{./images/}}

\newtheorem{theorem}{Theorem}
\newtheorem{assumption}[theorem]{Assumption}

\newtheorem{lemma}[theorem]{Lemma}
\newtheorem{remark}[theorem]{Remark}
\newtheorem{definition}[theorem]{Definition}
\newtheorem{proposition}[theorem]{Proposition}

\newenvironment{proof}{\par\noindent\emph{Proof. }}{\hfill$\square$\par}
\numberwithin{equation}{section}
\providecommand{\href}[2]{#2}

\newcommand{\beq}{\begin{equation}}
\newcommand{\eq}{\end{equation}}
\newcommand{\E}{\mathbb{E}}
\newcommand{\V}{\text{Var}}
\renewcommand{\P}{\mathbb{P}}

\newcommand{\e}{{\rm e}}
\renewcommand{\d}{{\rm d}}

\newcommand\blfootnote[1]{%
	\begingroup
	\renewcommand\thefootnote{}\footnote{#1}%
	\addtocounter{footnote}{-1}%
	\endgroup
}

\title{Pollaczek contour integrals for the \\fixed-cycle traffic-light queue\blfootnote{Department of Mathematics and Computer Science, Eindhoven University of Technology, P.O. Box 513, 5600MB Eindhoven, The Netherlands. Email: \href{mailto:m.a.a.boon@tue.nl}{m.a.a.boon@tue.nl}, \href{mailto:a.j.e.m.janssen@tue.nl}{a.j.e.m.janssen@tue.nl}, \href{mailto:j.s.h.v.leeuwaarden@tue.nl}{j.s.h.v.leeuwaarden@tue.nl} and \href{mailto:r.w.timmerman@tue.nl}{r.w.timmerman@tue.nl}}}

\author{M.A.A. Boon \and A.J.E.M. Janssen \and J.S.H. van Leeuwaarden \and R.W. Timmerman}

\begin{document}
	\maketitle

	\begin{abstract}
		The fixed-cycle traffic-light (FCTL) queue is the standard model for intersections with static signaling, where vehicles arrive, form a queue and depart during cycles controlled by a traffic light. Classical analysis of the FCTL queue based on transform methods requires a computationally challenging step of finding the complex-valued roots of some characteristic equation.  Building on the recent work of Oblakova et al.~\cite{ahmad}, we obtain a contour-integral expression, reminiscent of Pollaczek integrals for bulk-service queues, for the probability generating function of the steady-state FCTL queue. We also show that similar contour integrals arise for generalizations of the FCTL queue introduced in \cite{ahmad} that relax some of the classical assumptions. Our results allow to compute the queue-length distribution and all its moments using algorithms that rely on contour integrals and avoid root-finding procedures.

		
	\end{abstract}

	
	
	\section{Introduction}
	The fixed-cycle traffic-light (FCTL) queue is an intensively studied stochastic model in traffic engineering \cite{BvLfctlcorrelated,darroch,hagen1989comparison,mcneill,newell65,fctlsolo,webster}. Vehicles arrive to an intersection controlled by a traffic light and form a queue. The FCTL queue is traditionally modeled in discrete time and time is divided into slots of unit length. The green and red periods, of length $g$ and $r$ slots, respectively, and thus the cycle length $c=g+r$, are assumed to be fixed multiples of one slot.
	Each slot corresponds to the time needed for a delayed vehicle to depart from the queue. Vehicles that arrive during a red period are delayed, as well as those that arrive during a green period and meet a non-empty queue. Those vehicles that arrive to the queue and are delayed, join the queue at the end of the slot in which they arrive.
	Vehicles that arrive during a green period and meet no other vehicles in the queue, are treated according to the following assumption:
	
	\begin{definition}[FCTL assumption] For those cycles in which the queue clears before the green period terminates, all vehicles that arrive during the residual green period pass through the system and experience no delay whatsoever.
	\end{definition}
	
	\noindent So, in a case of a non-empty queue during the green period, the queue length in the next slot is reduced by one compared to the queue length in the previous slot, but increased by the number of arrivals in that slot, whereas the queue remains empty if the queue was already empty at the start of the slot.
	
	The FCTL assumption is quite realistic for straight-going traffic flows, since vehicles that find no queue during a residual green period will proceed without stopping and can keep driving at free-flow speed.
	For turning flows, however, the FCTL assumption seems too restrictive. Oblakova et al.~\cite{ahmad} considered several generalizations of the FCTL queue that alleviate the classical FCTL assumption to account for right-turns, disruptions of the traffic and uncertainty in departure times. We will cover these extensions in Subsection \ref{general}.
	
	The FCTL assumption (or its relaxed counterparts in \cite{ahmad}) causes some mathematical challenges. Let $X_{k,n}$ denote the queue length at time $k$ in cycle $n$ (time expressed in slots). Then, in cycle $n$, $X_{0,n}$ is the queue length at the beginning of the green period, and $X_{g,n}$ the overflow defined as the queue length at the end of the green period (and thus the beginning of the red period). Let $A_n$ denote the total number of vehicles that arrive at the intersection in between the two measurements of the overflow $X_{g,n}$ and $X_{g,n+1}$. Thus $A_n$ are the arrivals from the end of the time slot at which $X_{g,n}$ is measured onwards in a consecutive red and green period. 
	Further, $A_n=A_n^d+A_n^p$, where $A_n^d$ denotes the number of delayed vehicles and $A_n^p$ the number of vehicles that pass without delay on behalf of the FCTL assumption. The overflow queue can then be defined as
	\begin{equation}\label{overflow}
	X_{g,n+1}=\max\{X_{g,n}+A_{n}^d-g,0\}.
	\end{equation}
	The fact that $A_n^d$ depends on both $X_{g,n}$ and the exact specification of when the arrivals occur makes  (\ref{overflow}) hard to analyze.
	To capture that level of detail, let $Y_{k,n}$ denote the number of vehicles that arrive to the intersection during slot $k$ in cycle $n$. The random variables $Y_{k,n}$ are assumed independent and identically distributed, for all $k$ and $n$. Notice that all of the above assumptions together make that the queue lengths at the end of time slots can be modeled as a discrete-time Markov chain. Using analytic techniques suitable for dealing with such Markov chains, Darroch \cite{darroch} obtained the probability generating function (PGF) of the steady-state overflow queue (the number of vehicles waiting in front of the traffic light at the end of a green period) and the PGF of the steady-state delay was obtained in van Leeuwaarden \cite{fctlsolo}. Hence, all information about the distribution of the steady-state overflow queue and steady-state delay in the FCTL queue can be obtained from the results in \cite{darroch,fctlsolo}, including all moments of the steady-state queue length and delay, and the distribution of the output process (the way vehicles leave the intersection).
	
	The FCTL queue belongs to a large class of cyclic queueing models related to vehicle dispatching with uncertain arrivals and bulk service \cite{powell1985analysis,powell1986bulk,van2006discrete}. A range of transportation and manufacturing systems can be modeled in this way, including batch production systems, bulk movements of goods in a factory, truck shipments and bus transportation.  Within this class of models, many different rules can be considered that apply to customer arrivals and vehicle departures within a cycle. Think of vehicle-cancellation policies that hold a vehicle until the queue length reaches a specified threshold. The FCTL assumption can also be viewed as a special rule that influences the dynamics within a cycle.
	
	The classical FCTL treatment in \cite{darroch,fctlsolo} comes with computational challenges. This is because the PGF for the stationary queue length distribution contains $g-1$ boundary terms that need to be found separately. The traditional way of determining these remaining unknowns consists of two steps: finding the $g-1$ complex-valued roots and using these roots as input for a system of linear equations whose solution gives the boundary terms. Both steps can present difficulties, but were somehow considered unavoidable in the mathematically rigorous treatment of the FCTL queue \cite{darroch,mcneill,newell65,fctlsolo,webster} and of related bulk-service queues \cite{britt1,powell1985analysis,powell1986bulk}. This perspective changed with the work of Oblakova et al.~\cite{ahmad} that presented a contour integral expression of the mean stationary queue length. The present paper extends that methodology to the transform domain and results in a contour-integral expression for the PGF of the FCTL queue and some of the generalizations considered in \cite{ahmad}.
	
	\vspace{.3cm}
	\noindent{\bf Paper outline.} We  present in Section \ref{sec:main} the main result of the paper, an alternative representation for the transform solution in terms of a contour integral. Theorem \ref{mainttt} deals with the classical FCTL queue and Theorem \ref{thm:extension} with the generalized FCTL queues introduced in \cite{ahmad}.
	We also explain in Section \ref{sec:num}  how these contour integrals lead to algorithms that can compete with existing algorithms based on root-finding. The proof of our main results is presented in Section \ref{sec:proof} and uses several basic notions from complex analysis. We present some conclusions in Section \ref{sec:con}.

	\section{Main results}\label{sec:main}
	We briefly review the standard solution method for obtaining an exact transform solution for the steady-state FCTL queue in Subsection \ref{subb1}. We then present in Subsection \ref{subb2} the  contour integral representation for the FCTL queue and in Subsection \ref{general} for the generalized FCTL queues. 
	
	\subsection{Standard solution}\label{subb1}
	Let $Y$ be the number of arrivals during one slot and define $Y(z)= \E[z^Y]$. Assume $\P(Y=0)>0$, $Y'(1)<1$, and $Y(z)$ to be analytic in a region $|z|<R$ with $R>1$ and $R$ maximal. The key quantity in the mathematical analysis of the FCTL queue is the steady-state overflow queue, defined as $X_g=\lim_{n\to\infty}X_{g,n}$.
	Clearly, to have stability, and for $X_g$ to be well-defined,
	\begin{align}\label{stabilityc}
	c\E[Y]<g.
	\end{align}
	Using the kernel method and transform techniques, the PGF of $X_g$, denoted by $X_g(z)=\E[z^{X_g}]$ can be obtained using a by now classical line of reasoning. With $A(z)=\E[z^A]=Y(z)^c$ it can be shown that \cite{darroch,fctlsolo}
	\begin{equation}\label{pgf}
	X_g(z)=\frac{(z-Y(z))\sum_{k=0}^{g-1}q_kz^k Y(z)^{g-1-k}}{z^g-A(z)}.
	\end{equation}
	This expression still contains $g$ unknowns $q_0,\ldots,q_{g-1}$, representing the probability that the queue empties in slot $k$, thus $\P(X_k=0)=q_k$, which can be found by exploiting the analytic properties of PGFs. With Rouch\'{e}'s theorem, it can be shown that the
	denominator of (\ref{pgf}) has $g$ zeros on or within
	the unit circle $|z| = 1$.
	Because a PGF is
	well-defined in $|z|\leq 1$, the numerator of $X_g(z)$ should vanish at each of the zeros. This gives
	$g$ equations. One of the zeros equals 1, and leads to a trivial equation.
	However, the normalization condition $X_g(1)=1$ provides an
	additional equation.
	So that summarizes the highest level of general development for FCTL queue analysis: transform techniques yield an expression for $X_g(z)$ that in order to be evaluated demands finding $g-1$ roots in the complex plane of the function $z^g=A(z)$ and solving a set of $g$ linear equations.
	
	
	\subsection{Standard FCTL queue}\label{subb2}
	We now turn to the alternative expression for $X_g(z)$, where we change the argument of $X_g$ from $z$ to $w$ to distinguish between the standard version of the PGF and the contour-integral version.
	Here is the main result in this paper:
	\setcounter{theorem}{0}
	\begin{theorem}[Pollaczek integral for FCTL]\label{mainttt}
		There is an $\epsilon_0>0$ such that for all $\epsilon\in(0,\epsilon_0)$
		\begin{equation}\label{pk1}
		X_g(w)=\exp\left(\frac{1}{2\pi i}\oint_{|z|=1+\epsilon}\frac{Y'(z)z-Y(z)}{z-Y(z)}\,\frac{w-Y(w)}{z Y(w) -w Y(z)}
		\ln\left(1-\frac{A(z)}{z^g}\right)\,\d z\right), \qquad |w|<1+\epsilon,
		\end{equation}
		with principal value of the logarithm.
	\end{theorem}
	Here, $\epsilon_0$ should satisfy the inequality $\epsilon_0<\min\{t_0,R_0\}$, where $t_0 = \sup\{t\in\mathbb{R}_+|Y^\prime(t)t-Y(t)\leq 0 \}$ and $R_0$ is the unique root with smallest modulus of $A(z)=z^g$ in $(1,\infty)$. Since $\E[Y] <g/c<1$, this root always exists, also in case $Y(z)=y_0+y_1z$.
	\setcounter{theorem}{0}
	\begin{remark}
		The formula \eqref{pk1} for $X_g(w)$ is essentially equivalent with
		\begin{equation}\label{pk2}
		X_g(w)=\exp\left(\frac{1}{2\pi i}\oint_{|z|=1+\epsilon}\ln\left(\frac{wY(z)-zY(w)}{Y(z)-z}\right)\,\frac{(z^g-A(z))'}{z^g-A(z)}\,\d z\right),
		\end{equation}
		except that the validity range is more delicate due to the more complicated argument of the $\ln$ in~\eqref{pk2}. Formula \eqref{pk1} follows upon manipulating \eqref{pk2} using partial integration (details in Section \ref{sec:proof}).
	\end{remark}

	\vspace{.3cm}
	\noindent{\bf Sketch of the proof.} The proof of Theorem \ref{mainttt} finds a way to go from representation
	\eqref{pgf} to contour integrals. A significant start in this direction was made by \cite{ahmad}, who
	rewrote \eqref{pgf} as
	\begin{equation}\label{pgf2}
	X_g(z)=\frac{(z-Y(z))z^{g-1}\sum_{k=0}^{g-1}q_k \left(\frac{Y(z)}{z}\right)^{g-1-k}}{z^g-A(z)}.
	\end{equation}
	Then denote the $g$ roots of $z^g=A(z)$ on and within the unit circle by $z_0=1,z_1,\ldots,z_{g-1}$.
	Now here is where the authors in \cite{ahmad} took an eye-opening step: instead of using the $g$ roots in the traditional manner for finding the unknowns $q_k$ and completing the transform \eqref{pgf}, use these roots for factorizing the numerator of \eqref{pgf}.
	Notice that this cannot be done immediately, because interpreted as a function of $z$, the numerator is by no means a polynomial of degree $g$ or less. However, by treating the function $Y(z)/z$ as a variable itself, the summation in the numerator is a polynomial of degree $g-1$ and can be factorized as
	\begin{equation}\label{pgf3}
	\sum_{k=0}^{g-1}q_k \left(\frac{Y(z)}{z}\right)^{g-1-k}=q_0\prod_{k=1}^{g-1}\left(\frac{Y(z)}{z}-\frac{Y(z_k)}{z_k}\right),
	\end{equation}
	using that $X_g(z)$ is well-defined in the disk $|z|\leq1$, that $z_1,...z_{g-1}$ are roots of the denominator and therefore also should be roots of the numerator, and that $Y(z)/z$ is injective (see Section \ref{sec:proof}).
	After normalization using $X_g(1)=1$ the factorization in \eqref{pgf3} leads to the representation
	\begin{equation}\label{pgf4}
	X_g(z)=\frac{g-A'(1)}{z^g-A(z)}\cdot \frac{z-Y(z)}{1-Y'(1)} \cdot z^{g-1}\prod_{k=1}^{g-1}\frac{Y(z)/z-Y(z_k)/z_k}{1-Y(z_k)/z_k}.
	\end{equation}
	Our proof then proceeds by interpreting \eqref{pgf4} as the outcome of Cauchy's residue theorem, the classical tool from complex analysis to evaluate line integrals of analytic functions over closed curves. An important step is to write
	\begin{equation}\label{pgf5}
	\ln \left(z^{g-1}\prod_{k=1}^{g-1}\frac{Y(z)/z-Y(z_k)/z_k}{1-Y(z_k)/z_k}\right)=\sum_{k=1}^{g-1}\ln\left(\frac{zY(z_k)-z_k Y(z)}{Y(z_k)-z_k}\right),
	\end{equation}
	and to regard \eqref{pgf5} as the sum of residues at $z=z_k$. To construct an analytic function that, in conjunction with Cauchy's theorem and the closed curve $|z|=1+\epsilon$, returns \eqref{pgf5} and has singularities at $z_1,\ldots,z_{g-1}$, leads us to consider the integrand in \eqref{pk2}. Here, the logarithmic function
	\begin{equation}\label{logg}
	\ln\left(\frac{wY(z)-zY(w)}{Y(z)-z}\right)
	\end{equation}
	follows from \eqref{pgf5} and the singularities with appropriate residues are created through $(z^g-A(z))'/(z^g-A(z))$.
	After careful consideration of the analytic properties of the integrand in \eqref{pk2}, we then show that Cauchy's theorem gives \eqref{pgf4} from which \eqref{pk2} follows. As said, \eqref{pk1} is obtained by manipulating \eqref{pk2}, using partial integration.
	The formal proof of Theorem \ref{mainttt} presented in Section \ref{sec:proof} contains several challenging steps, and requires among other things a proof that the function $Y(z)/z$ is injective in a region that contains the unit disk, and a way to account for the branch cut caused by the logarithm in \eqref{logg} being taken over negative values.
	
	\vspace{.3cm}
	\noindent{\bf Historical notes.} Integrals of this sort go a long way back in the history of queueing theory and were first found in the ground-breaking work of Pollaczek on the classical single-server queue (see \cite{awp,cohen,janssen2008back} for historical accounts). Let us point out the connection to the well known Pollaczek type integral for the discrete bulk-service queue \cite{britt1}, a discrete-time queueing model in which customers upon arrival are placed in a queue and after some stochastic service time, a (possibly stochastic) number of customers is served, all at once. Served customers leave the system immediately, whereas the remaining customers in the queue have to wait at least one more service period. The analysis of the bulk-service queue is easier than that of the FCTL queue. To see this, observe that the analysis of the FCTL queue is greatly simplified if all vehicles were delayed \cite{broek}, so that all vehicles arrive while
	the queue length is at least one and the complicated $A_n^d$ random variable in \eqref{overflow}
	can be replaced by $A_n$. In that case \eqref{overflow} becomes a standard stochastic recursion driven by
	i.i.d.~random variables and the FCTL queue reduces to the classical bulk-service queue, a special case of the more general single-server queue investigated by Pollaczek.
	Let $X_b$ denote the steady-state queue length in that bulk-service queue, defined as the solution of the stochastic equation
	\begin{equation}X_b\stackrel{d}{=}\max\{X_{b}+A-g,0\}.
	\end{equation}
	Pollaczek's result then says that (see \cite{janssen2018spitzer,britt1} for a direct derivation)
	\begin{equation}\label{pk}
	X_b(w)=\exp\left(\frac{1}{2\pi i}\oint_{|z|=1+\epsilon}\ln\left(\frac{w-z}{1-z}\right)\,\frac{(z^g-A(z))'}{z^g-A(z)}\,\d z\right)
	\end{equation}
	holds when $|w|<1+\epsilon$ with $\epsilon$ positive and bounded by some constant. Observe the striking similarity with \eqref{pk2}. While the FCTL queue is harder to analyze than the bulk-service queue, the two contour-integral representations \eqref{pk2} and \eqref{pk} only differ in the logarithmic function. The authors find this quite surprising  themselves, particularly because there seems no way to interpret the FCTL queue as a reflected random walk (that is, a recursive structure with i.i.d.~increments), while in the literature so far this seems to be a prerequisite for establishing Pollaczek-type contour integrals. Do observe that \eqref{pk} is valid in an area that includes the unit disk while \eqref{pk2} is guaranteed only in an open set containing $[0,1]$, see Section \ref{sec:proof}. This objection does not hold against the representation \eqref{pk1} of $X_g(w)$.
	
	
	\begin{remark}The bulk-service queue serves as a popular approximation of the FCTL queue {\normalfont\cite{broek}}. In fact, for Bernoulli arrivals with per time slot one or no arrival (which is case (i) in the proofs of Lemmas \ref{lem:injective} and \ref{lem:injective2}), this approximation becomes exact. To see this, substitute $Y(z)=1-p+pz$ into the logarithmic function in \eqref{pk2}, and observe that this gives the logarithmic function in \eqref{pk}. Obviously, when $Y$ can take values larger than one, the bulk-service queue is an approximation and yields an upper bound on the overflow queue.
	\end{remark}

	\subsection{Generalized FCTL queues} \label{general} Oblakova et al.~\cite{ahmad} have introduced generalized FCTL queues, and established contour integrals for the first moment of the steady-state queue length. We now show how contour integral representations for these generalized FCTL queues follow almost directly from the standard FCTL queue. We start from the definition of $X(z)$ in \cite{ahmad}, a generalization of the function $X_g(z)$ that contains as special cases several extensions of the FCTL queue.
	
	\setcounter{theorem}{1}
	\begin{definition}[Oblakova et al.~\cite{ahmad}]\label{as} Consider the function $X(z)$ with $X(1)=1$ and
		\begin{equation}\label{eq:Oblakova}
		X(z) = \frac{\sum_{k=0}^{g-1}x_kz^kB(z)^{g-1-k}}{z^g-A(z)}\xi(z),
		\end{equation}
		where $B(z)$ and $A(z)$ are PGFs and $\xi(z)$ is a function satisfying $\xi(1)=0$, $\xi(z_l)\neq0$ with $z_l\neq1$ the roots of $z^g-A(z)$ inside the unit disk. Assume moreover that $B^\prime(1)<1$; $A^\prime(1)<g$; that for some $\delta>0$ the functions $A(z)$ and $B(z)$ are analytic within the disk $|w|<1+\delta$; and that $X(z)$ is analytic inside the unit disk and continuous up to the unit circle. Also assume that $t_0>1$, where $t_0 = \sup\{t\in\mathbb{R}_+|B^\prime(t)t-B(t)\leq 0 \}$.
	\end{definition}
	Here is the main result for the function $X(z)$:
	\setcounter{theorem}{1}
	\begin{theorem}[Pollaczek integrals for generalized FCTL queues]\label{thm:extension} Under Definition \ref{as}
		\noindent there exists an $\epsilon_0>0$ such that for all $\epsilon\in(0,\epsilon_0)$
		\begin{equation}\label{eq:thm}
		X(z)=\exp\left(\frac{1}{2\pi i}\oint_{|w|=1+\epsilon}\ln\left(\frac{zB(w)-wB(z)}{B(w)-w}\right)\,\frac{(w^g-A(w))'}{w^g-A(w)}\,\d w\right)\frac{1-B^\prime(1)}{z-B(z)}\frac{\xi(z)}{\xi^\prime(1)},
		\end{equation}
		for all $|z|<1+\epsilon$, with principal value of the logarithm.
	\end{theorem}
	
	\begin{proof}
		We shall express \eqref{eq:Oblakova} as a product of the PGF of the standard FCTL queue and some analytic function.
		Denote the $g-1$ roots of $z^g-A(z)$ inside the unit circle by  $z_1,...,z_{g-1}$. We can then rewrite \eqref{eq:Oblakova}, using $X(1)=1$, as
		\begin{equation}
		X(z) = \frac{(g-A^\prime(1))}{z^g-A(z)} \frac{f(z)}{f^\prime(1)}\prod_{k=1}^{g-1} \frac{B(z)z_k-zB(z_k)}{B(z_k)-z_k}\label{eq:oblakovaREW2}.
		\end{equation}
		Setting $B(z)=Y(z)$, $A(z)=Y(z)^c$ and $f(z)=z-Y(z)$ we see from \eqref{pgf4} that
		\begin{equation}\label{eq:decomp}
		X(z) = X_g(z)\frac{1-B^\prime(1)}{z-B(z)}\frac{f(z)}{f^\prime(1)}.
		\end{equation}
		This gives the result.
	\end{proof}
	\noindent Let us now discuss the extensions contained in $X(z)$.
	
	(i) The first extension concerns the right-turning traffic, in which case the difference in discharge rate between delayed and non-delayed vehicles almost vanishes. This requires to modify the FCTL assumption in order to put an upper bound on the number of vehicles that pass the traffic light without delay. This upper bound is set to one, so that at most one vehicle can depart per green slot. Following \cite{ahmad}, it can be shown that this model for right-turning traffic follows by setting $B(z)=Y(z)$, $A(z)=Y(z)^c$ and $\xi(z)=(z-1)Y(0)$, and the contour integral expression for the PGF thus follows from Theorem \ref{thm:extension}.
	
	(ii) Another extension of the classical FCTL queue is one that accounts for disruptions of the traffic flow by e.g.~pedestrians. To account for these disruptions, one could extend the red period or shorten the green period for the main stream of vehicles \cite{ahmad}. This extension thus requires a FCTL queue with random (but finite) green and red times, for which we choose $g=G$ with $G$ denoting the maximum green time. Setting $B(z)=Y(z)$, $A(z)=\sum_{r,g}p_{r,g}Y(z)^{r+g}z^{G-g}$ with $p_{r,g}$  the probability that a cycle consists of $g$ green and $r$ red slots, and $\xi(z)=z-Y(z)$, then shows that also this extension of interrupted flows is contained in Theorem \ref{thm:extension}.
	
	(iii) The third extension we mention relates to uncertainty in departure times of vehicles, for instance due to distracted drivers facing a green light, and hence causing the driver to not depart from the queue with some probability $p$  \cite{ahmad}. Assuming a geometrically number of slots before a driver leaves the queue models this situation, with $B(z)=Y(z)(1-p+pz)$, $A(z)=Y(z)^c(1-p+pz)^g$ and $\xi(z)=z-Y(z)(1-p+pz)$.
	
	(iv) A fourth extension deals with relaxing the independence assumption of the arrival process during the red slots \cite{ahmad}. In this extension, the arrivals during a red time within a cycle may be dependent (the arrivals during green slots still need to be~i.i.d.). For this FCTL queue we should choose $B(z)=Y(z)$, $A(z)=A_r(z)Y(z)^g$, where $A_r(z)$ denotes the PGF of the arrival process during the whole red period, and $\xi(z)=z-Y(z)$.
	
	The present paper adds to \cite{ahmad} the PGFs in terms of contour integrals, of which the contour integrals for the mean queue length obtained in \cite{ahmad} follow by evaluating the derivative in one. For insights into the differences between the various FCTL queue extensions we refer to the elaborate numerical study in \cite{ahmad}.
	
	

	\section{Algorithmic methods}\label{sec:num}
	
	We now discuss the computational challenges that come with calculating the steady-state queue length distribution, using either the contour integrals in Theorems \ref{mainttt}  and  \ref{thm:extension} or the standard expression in terms of roots.
	The algorithms using contour integrals in this section are based on the representation \eqref{pk2} (but one could also take \eqref{pk1}). Notice that we only need to expand $X_g(w)$ at $w=0$ and $w=1$, so inside the validity range of \eqref{pk2}.
	
	%

	\subsection{From PGF to performance measures}
	The mean stationary overflow queue $\E X_g$ is given by $X_g'(1)$ and takes the form
	\begin{align} \label{e7}
	\E X_g&=\frac{1}{2\pi i}\oint_{|z|=1+\epsilon}\frac{Y(z)-z Y'(1)}{Y(z)-z }~\frac{(z^g-A(z))'}{z^g-A(z)}\,\d z.
	\end{align}
	This result was recently obtained in \cite{ahmad} using a direct proof that converted the classical expression for $\E X_g$ in terms of complex-valued roots into the integral expression \eqref{e7}.
	
	From the PGF $X_g(z)$ we can in principle determine all stationary moments.
	Define
	\begin{align}
	f(w)&:=\frac{1}{2\pi i}\oint_{|z|=1+\epsilon}g(w,z)\,\frac{(z^g-A(z))'}{z^g-A(z)}\,\d z,\\
	g(w,z)&:=\ln\left(\frac{wY(z)-zY(w)}{Y(z)-z}\right),\\
	h_k(w) &:= \begin{cases}
	1 & \qquad k = 0,\\
	h_{k-1}(w)f'(w)+h_{k-1}'(w) & \qquad k = 1, 2, \dots
	\end{cases}
	\end{align}
	The moments $\E[X_g^k]$ then follow from symbolically differentiating the PGF  \eqref{pk2}, and these derivatives can be expressed as
	\begin{equation}
	X_g^{(k)}(w) := \frac{\text{d}^k}{\text{d} w^k} X_g(w) =\frac{\text{d}^k}{\text{d} w^k} \exp\left(f(w)\right) = h_k(w)\exp\left(f(w)\right),
	\end{equation}
	for $k=0, 1,2,\dots$. Using this recursive expression, $X_g^{(k)}(w)$ can be expressed in terms of $f(w)$ and the first $k$ derivatives of $f(w)$, denoted by $f^{(1)}(w), \dots, f^{(k)}(w)$ with
	\begin{align}
	f^{(j)}(w)&:=\frac{\partial^j}{\partial w^j}  \frac{1}{2\pi i}\oint_{|z|=1+\epsilon}g(w,z)\,\frac{(z^g-A(z))'}{z^g-A(z)}\,\d z\nonumber\\
	&=\frac{1}{2\pi i}\oint_{|z|=1+\epsilon}g^{(j)}(w,z)\,\frac{(z^g-A(z))'}{z^g-A(z)}\,\d z
	\end{align}
	and $g^{(j)}(w,z):=\frac{\partial^j}{\partial w^j}g(w,z)$, for $j=1,2,\dots,k$.
	After substituting $w=1$, we can express the first $k$ moments of $X_g$ in terms of $k$ contour integrals that only involve the model primitives and the first $k$ moments of $Y$. Using $f(1)=0$, the variance of $X_g$ given by $\V(X_g)=h_2(1)+h_1(1)-(h_1(1))^2$ takes the form
	\begin{align} \label{e7v}
	\V(X_g)&=\frac{1}{2\pi i}\oint_{|z|=1+\epsilon}\frac{z^2\V(Y)-zY(z)(1+\E (Y^2)-2\E Y)}{(z-Y(z))^2}~\frac{(z^g-A(z))'}{z^g-A(z)}\,\d z.
	\end{align}

	To determine the stationary distribution of the overflow queue we use that
	\begin{equation}
	\P(X_g=k) = \frac{1}{k!}\frac{\text{d}^k}{\text{d} w^k} X_g(w)\Big|_{w=0}=\frac{1}{k!}h_k(0)\exp\left(f(0)\right).
	\end{equation}
	First observe that
	\begin{equation}
	\P(X_g=0) = \exp\left(f(0)\right) = \exp\left(\frac{1}{2\pi i}\oint_{|z|=1+\epsilon}\ln\left(\frac{z \P(Y=0)}{z-Y(z)}\right)\frac{(z^g-A(z))'}{z^g-A(z)}\,\d z\right).
	\end{equation}
	Expressions for the other probabilities $\P(X_g=k)$ follow in a similar way, but require evaluating the resulting function at $w=0$ instead of $w=1$ and dividing by $k!$. As a consequence, $\P(X_g=k)$ can be expressed in terms of $f(0), f^{(1)}(0), \dots, f^{(k)}(0)$, again an expression that involves explicit contour integrals only.

	
	{
		
		{
			
			\subsection{Roots or integrals?}\label{sec:examples}
			Compared with root finding, contour integrals have advantages and disadvantages. On the one hand, avoiding the implicitly defined roots is nice, because the integrals are explicit expressions in terms of the model primitives $g$, $r$  and $Y(z)$. On the other hand, the number of terms required to evaluate $f^{(j)}(w)$ grows exponentially in $j$. For tail probabilities this symbolic differentiation becomes computationally cumbersome.
			
			While in the early queueing literature root finding was considered to be prohibitively difficult, with the computational methods available nowadays it is possible to find the complex-valued roots of $z^g-A(z)$ with great accuracy. In Appendix \ref{app:alg} we present the root-finding algorithm that we use in this paper, which after extensive testing was found to be accurate and reliable for all choices of $A(z)$. The simple idea behind the algorithm is to approximate $A(z)$ with its Taylor series $A_n(z)$ of order $n$, reducing the problem to finding roots of polynomial equations, and subsequently use Newton's method to find the roots of $z^g-A(z)$ with arbitrary precision. We also present some results that show that the roots of the $n$-th system converge to the roots of $z^g-A(z)$, and provide an explicit characterization of the roots for the case when $A(z)$ is the PGF of a Poisson random variable. In that case, the roots can be written in terms of the Lambert W-function.
			
			Extensive tests with both algorithms did not result in any numerical issues, except for two obvious limitations: for tail probabilities the symbolic differentiation within the integrand becomes a bottleneck, and for root finding loss of accuracy is expected when the number of roots $g$ becomes excessively large (although a thousand roots present no difficulties). For most if not all practical purposes both methods lead to reliable and accurate algorithms.
			
			In terms of computation time, contour integration is generally slower than root finding. For moments there is little difference, because both methods lead to explicit expressions. For queue-length probabilities, however, numerical inversion of the PGF is required (see e.g.~\cite{abate1992numerical}), and in that case calculating the many contour integrals becomes computationally expensive. 

			\begin{figure}[t]
				\centering
				\begin{tabular}{cc}
					\includegraphics[width=0.47\textwidth]{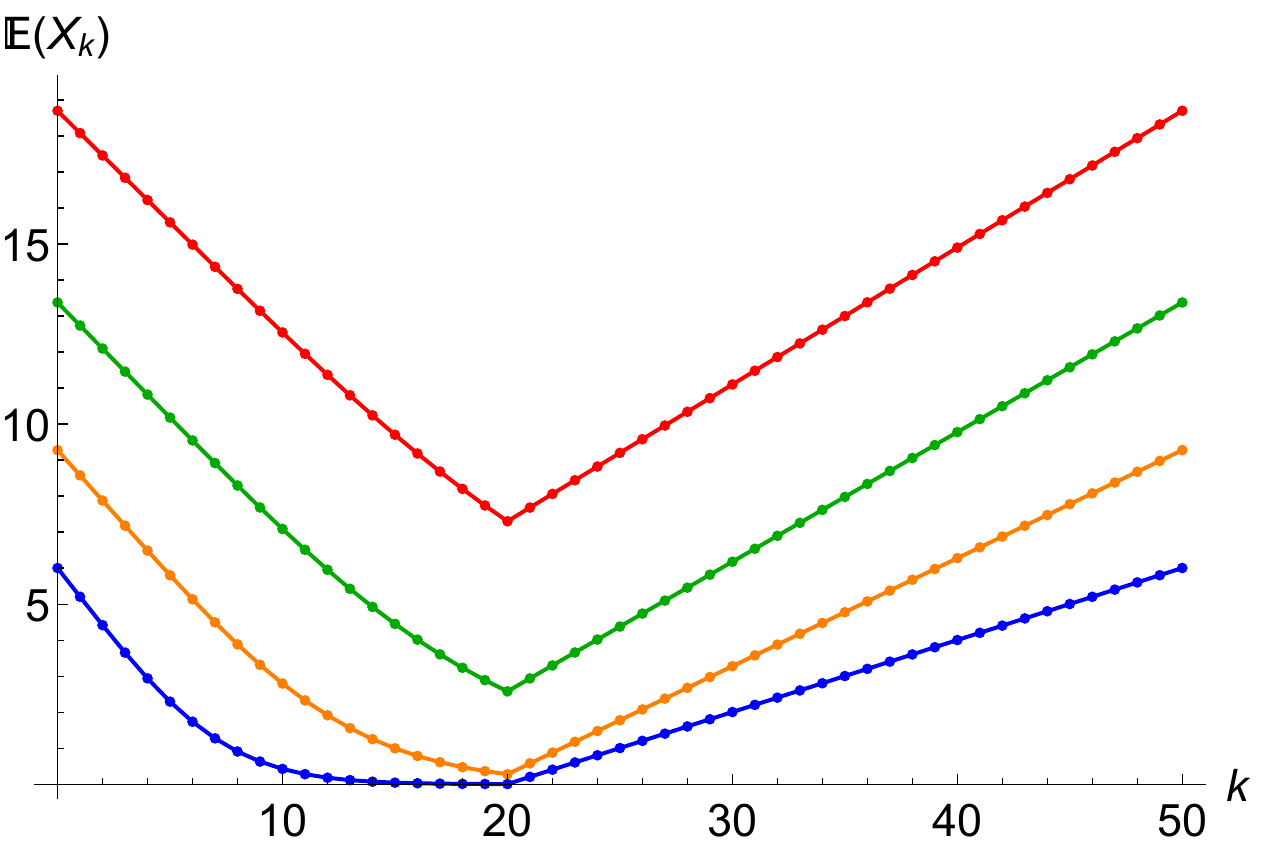}
					&
					\includegraphics[width=0.47\textwidth]{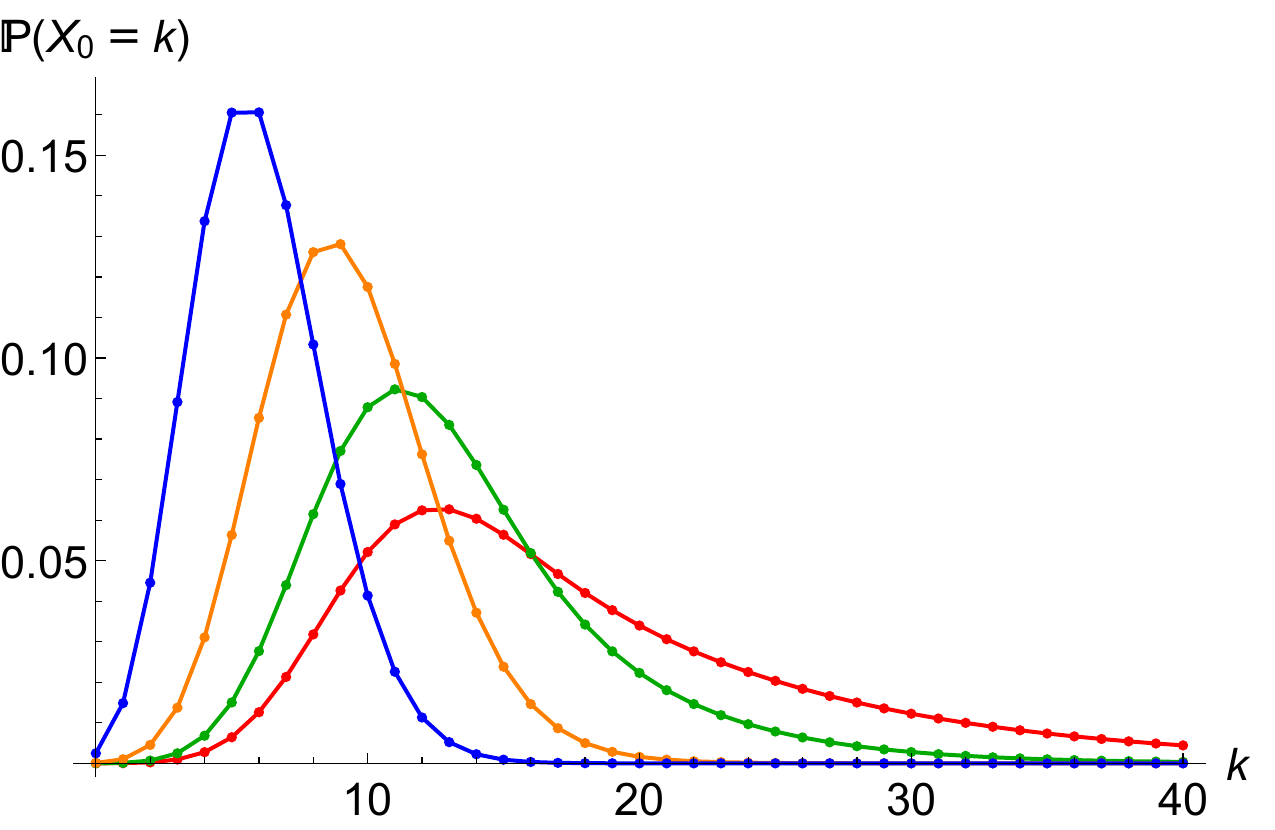}\\
					(a) & (b) \\[1ex]
					\includegraphics[width=0.47\textwidth]{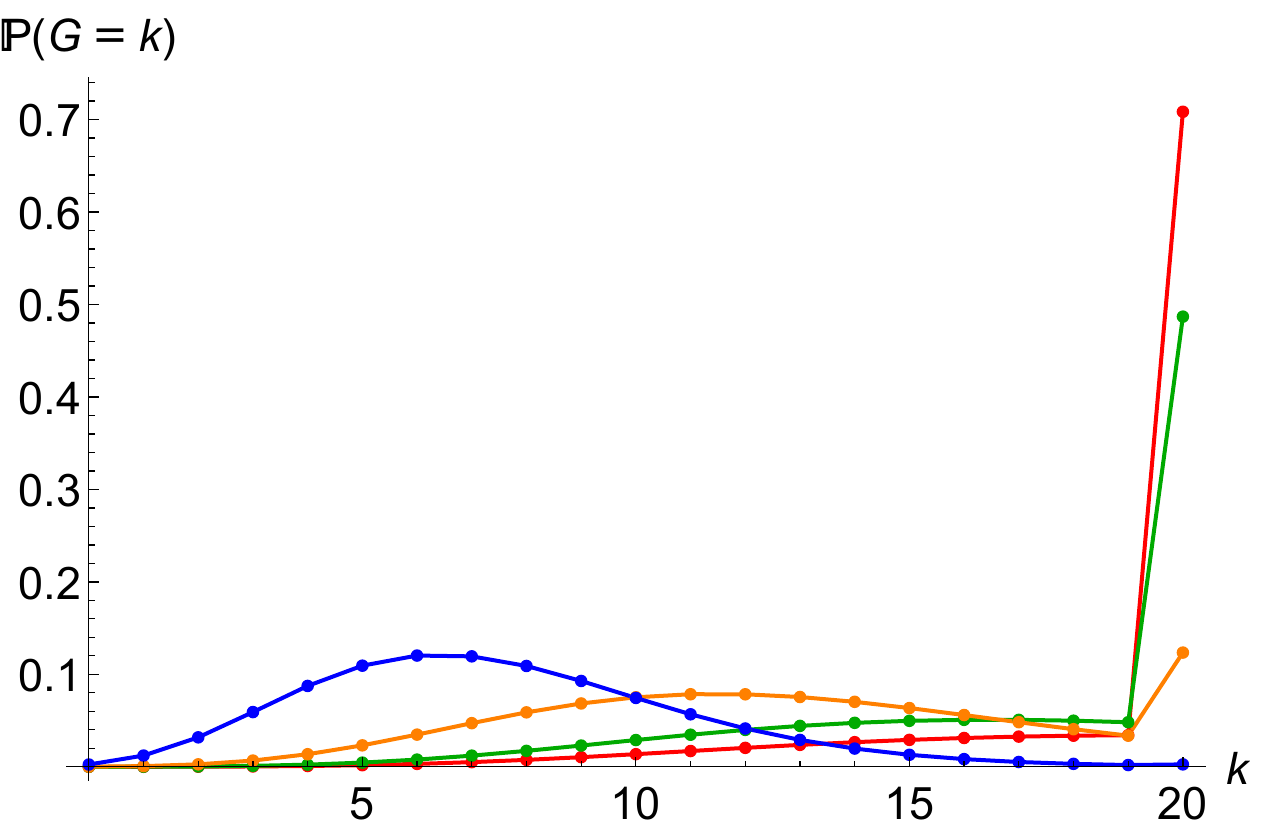}
					&
					\includegraphics[width=0.47\textwidth]{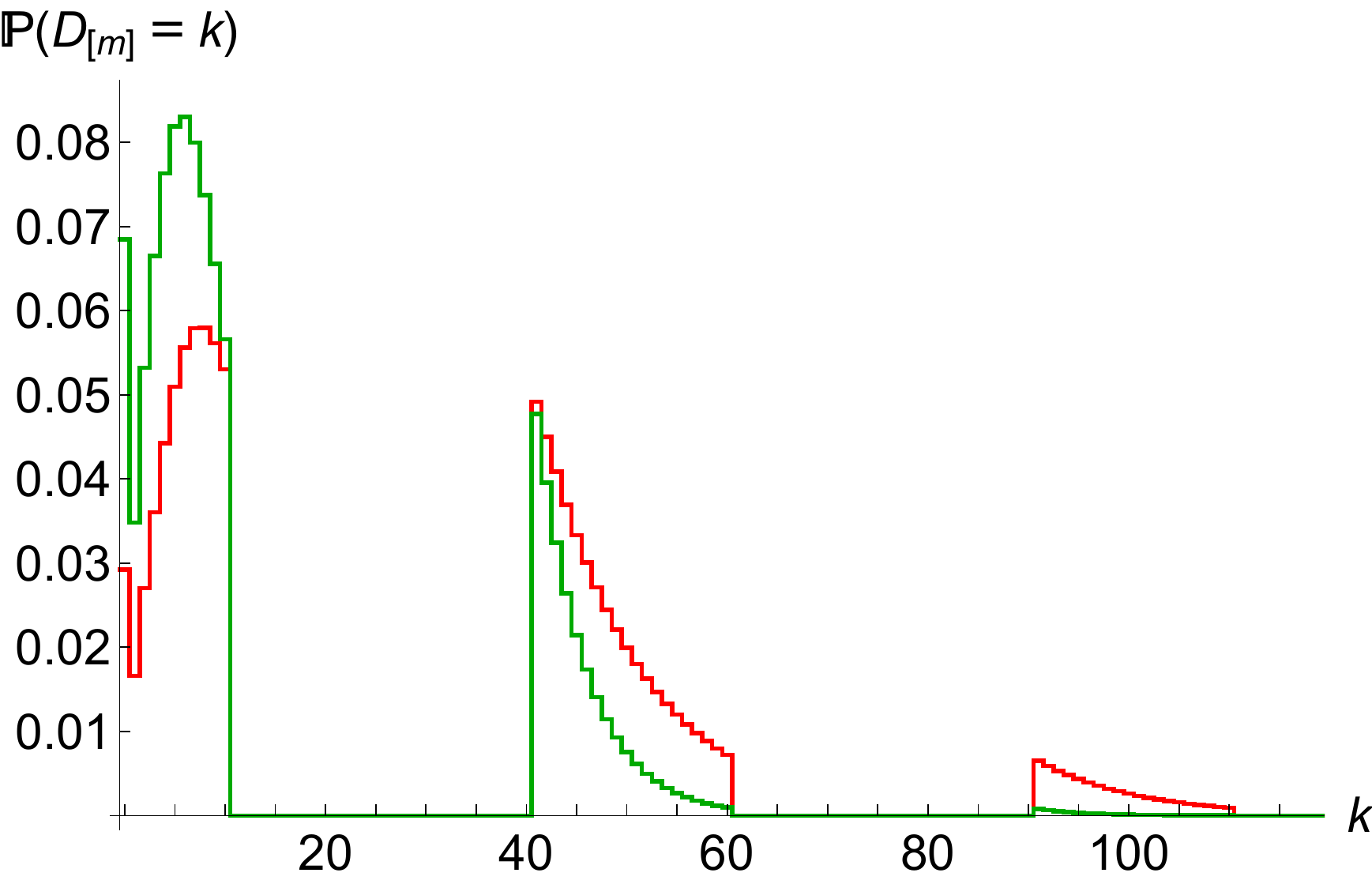}\\
					(c)  & (d)
				\end{tabular}
				\caption{Several performance measures for the FCTL queue in Section \ref{sec:examples} with $g=20$, $r=30$ and Poisson arrivals. The colors blue, orange, green and red correspond to volume over capacity ratios of, respectively, 0.5, 0.75, 0.9 and 0.95. The subfigures show (a) mean queue lengths during a cycle, (b) the queue length distribution at the start of green periods, (c) effective green periods and (d) delay distribution of vehicles arriving in slot 10 (for $\rho=0.9,0.95$ only).}
				\label{fig:numex}
			\end{figure}
			
			To illustrate the algorithms we now show some results for the FCTL queue with $g=20$ and $c=50$. We consider Poisson arrivals with on average $\lambda$ vehicles arriving per slot, and four scenarios: $\lambda=0.2$ (light traffic), $\lambda=0.3$ (moderate traffic), $\lambda=0.36$ (heavy traffic), and $\lambda=0.38$ (extreme traffic). These arrival rates correspond to a volume/capacity ratio $\rho=\lambda c /g$ ranging from 0.5 to 0.95. The results are calculated with both roots and contour integrals, and are on the scale of the displayed figures indistinguishable.

			
			Figure \ref{fig:numex}(a) shows the mean queue lengths $E[X_0], \dots, \E[X_c]$ through one cycle. Observe the strong cyclic behavior and the high sensitivity for $
			\rho$.
			Figure \ref{fig:numex}(b) shows the queue-length distribution at cycle start, the moment that the traffic signal turns green and queue lengths are expected to peak.
			Observe the difference between operating at 75\% or 90\% of maximal capacity: the probability that more than 20 vehicles are waiting is only $0.002$ for $\lambda=0.3$ and $0.32$ for $\lambda=0.38$.
			Figure \ref{fig:numex}(c) depicts the distribution of the effective green time $G$, defined in \cite{broek} as the number of slots used for departure of delayed vehicles that arrive throughout the whole cycle. We have
			\begin{equation}
			\P(G=k) = \begin{cases}
			q_0 & \qquad \text{ for }k=0,\\
			q_k - q_{k-1} & \qquad \text{ for }k=1, \dots, g-1,\\
			1 - q_{g-1} & \qquad \text{ for }k=g.
			\end{cases}
			\end{equation}
			Since only one delayed vehicle departs per slot, this can also be considered to be the distribution of the platoon length consisting of delayed vehicles departing during one cycle. Observe that $\P(G=g)$ is practically zero when $\rho=0.5$, but as high as $0.71$ when $\rho=0.95$, which means that only in 29\% of the cycles the green time is long enough to let the queue vanish.
			
			Finally, we consider the delay distribution of an arbitrary vehicle arriving in the 10-th slot, which is during the green period. The stationary delay of a vehicle arriving in slot $k$, denoted by $D_{[k]}$, is defined as the number of slots between arrival and departure, not including the slot of arrival. Figure \ref{fig:numex}(d) shows
			$D_{[10]}$, which can be computed directly from $X_{9}$, i.e., the number of vehicles waiting at the start of the 10-th slot. If $X_{9}=0$ we have that $D_{[10]}=0$; otherwise the delay can be expressed as a function of the number of vehicles present at the arrival of the tagged vehicle. This function (studied in detail in \cite{fctlsolo}) should take into account interruptions due to red periods, which explains the fragmented histograms in
			Figure \ref{fig:numex}(d).
			
			\section{Proof of the Pollaczek contour-integral representation}\label{sec:proof}
			As explained briefly in Section \ref{sec:main}, the proof of Theorem \ref{mainttt}
			exploits the factorized form \eqref{pgf4} and investigates in detail the logarithmic function
			\eqref{logg}. We present some useful properties of the function $Y(z)/z$, visible in both \eqref{pgf4} and \eqref{logg}. We then proceed to use Cauchy's theorem to obtain the contour-integral representation \eqref{pk2} for the case that $1<w<1+\epsilon$, and finally manipulate \eqref{pk2} to obtain \eqref{pk1} on the full range $|w|<1+\epsilon$.
			
			\subsection{Auxiliary results}
			Before we prove Theorem \ref{mainttt} we present some auxiliary results for the function $Y(z)/z$.
			In \cite[Theorem 1]{ahmad} it was shown that the function $Y(z)/z$ is injective on the disk $|z|\leq 1$, so that all 
			$Y(z_k)/z_k \neq Y(z_l)/z_l$ when $z_k\neq z_l$.
			For our proof we also need injectivity, but then for the larger disk with radius $t_0>1$. More specifically,
			let
			\begin{align}\label{eqnA3}
			t_0 :&= \sup\{t\in(0,R)\,\,|\,\,Y'(t)t-Y(t)\leq 0\},
			\end{align}
			where $R$ is the maximum value such that $Y(z)$ is analytic in the region $|z|<R$.
			
			\begin{lemma}\label{lem:injective}
				The function $Y(t)/t$ is strictly decreasing in $t\in(0,t_0]$.
			\end{lemma}
			
			\begin{proof}
				We have that
				\begin{equation}\label{eqnA1}
				\frac{Y(t)}{t}=\frac{y_0}{t}+y_1+y_2t+\dots \,,\qquad 0<t<R,
				\end{equation}
				is strictly convex since $y_0>0$, with derivative
				\begin{equation}\label{eqnA2}
				\left(\frac{Y(t)}{t}\right)'=\frac{Y'(t)t-Y(t)}{t^2},\qquad 0<t<R.
				\end{equation}
				
				Since $Y'(1)<Y(1)=1$, we have that $t_0>1$.
				Now consider the following cases: (i) $y_k=0$ for $k=2,3,\dots$, (ii) there is a $k=2,3,\dots$ such that $y_k\neq 0$. For case (i) $Y(t)/t =y_0t^{-1}+y_1$ is strictly decreasing in $t>0$ since $y_0>0$.
				For case (ii), $y_k>0$ for some $k\geq2$, and so
				\begin{equation}\label{eqnA4}
				Y'(t)t-Y(t)=-y_0+\sum_{k=2}^\infty (k-1)y_kt^k
				\end{equation}
				is strictly increasing in $t\in(0,R)$. From the definition of $t_0$, we then get that
				\begin{equation}\label{eqnA5}
				Y'(t)t-Y(t)<0, \quad t\in(0,t_0),
				\end{equation}
				and so $Y(t)/t$ is strictly decreasing in $t\in(0,t_0)$ by \eqref{eqnA2}.
			\end{proof}

			\begin{lemma}\label{lem:injective2}
				The function $Y(z)/z$ is injective on the open disk $|z|<t_0$, so that for $|z|<t_0$, $|w|<t_0$
				\begin{equation}\label{eqnA6}
				\frac{Y(z)}{z}=\frac{Y(w)}{w} \Rightarrow z=w.
				\end{equation}
			\end{lemma}
			\begin{proof}
				In case (i), $y_k=0$ for $k=2,3,\dots$, we have $Y(z)/z=y_0z^{-1}+y_1$ and the result is trivial since $y_0>0$. For case (ii), there is a $k=2,3,\dots$ such that $y_k\neq0$, we let $|z|<t_0$, $|w|<t_0$. Then
				\begin{align}\label{eqnA7}
				\left|\frac{Y(z)}{z}-\frac{Y(w)}{w}\right| &= \Big|y_0\frac{w-z}{zw}+\sum_{k=2}^\infty y_k(z^{k-1}-w^{k-1})\Big|\nonumber\\
				&= |z-w|\,\Big|-\frac{y_0}{zw}+\sum_{k=2}^\infty y_k\frac{z^{k-1}-w^{k-1}}{z-w}\Big|.
				\end{align}
				Let $t:=\max\{|z|, |w|\}<t_0$. Then $|y_0/(zw)|\geq y_0/t^2$ while
				\begin{equation}\label{eqnA8}
				\Big|\frac{z^{k-1}-w^{k-1}}{z-w}\Big|=\left|z^{k-2}+wz^{k-3}+\dots+zw^{k-3}+w^{k-2}\right|\leq(k-1)t^{k-2}.
				\end{equation}
				Therefore, when $z\neq w$,
				\begin{align}\label{eqnA7b}
				\left|\frac{Y(z)}{z}-\frac{Y(w)}{w}\right| \geq |z-w|\Big(\frac{y_0}{t^2}-\sum_{k=2}^\infty(k-1)y_kt^{k-2}\Big)>0
				\end{align}
				by \eqref{eqnA4} and \eqref{eqnA5}. This proves \eqref{eqnA6}.
			\end{proof}

			
			\begin{lemma}\label{lem:nozero}
				Let $\epsilon>0$ be such that $1+\epsilon < t_0$, and take $w \in (1, 1+\epsilon)$. For $|z|<t_0$,
				\begin{equation}\label{eqn210}
				\frac{w Y(z)-z Y(w)}{Y(z)-z}\in (-\infty, 0] \, \Leftrightarrow\, 1\leq z\leq w,
				\end{equation}
				i.e., the only $z$ in the disk with radius $t_0$ for which the fraction in \eqref{eqn210} is negative satisfy $1\leq z\leq w$.
				
				Furthermore
				\begin{equation}\label{eq:greaterzero}
				-1 < z < 1\, \Rightarrow\, \frac{wY(z)-zY(w)}{Y(z)-z}>0.
				\end{equation}
			\end{lemma}
			\begin{proof}
				For $a\leq 0$,
				\begin{equation}\label{eqn42}
				\frac{w Y(z)-z Y(w)}{Y(z)-z}=a \, \Leftrightarrow\, \frac{Y(z)}{z}=\frac{Y(w)-a}{w-a}.
				\end{equation}
				Since $1<Y(w)<w$, the function $(Y(w)-a)/(w-a)$ increases from $Y(w)/w$ at $a=0$ to 1 at $a=-\infty$ when $a$ decreases from 0 to $-\infty$. Since $Y(v)/v$ decreases strictly in $v\in [1,w]$, there is for any $a\leq 0$ a unique $v=v(a) \in [1,w]$ such that
				\begin{equation}\label{eqn43}
				\frac{Y(v)}{v}=\frac{Y(w)-a}{w-a}.
				\end{equation}
				Since by Lemma \ref{lem:injective2} $Y(z)/z$ is injective in $|z|<t_0$ we get \eqref{eqn210}.
				
				We next show \eqref{eq:greaterzero}. Obviously, \eqref{eq:greaterzero} holds for $z=0$. For $z\neq0$, we have
				\begin{equation}
				\frac{wY(z)-zY(w)}{Y(z)-z} = \frac{Y(z)/z-Y(w)/w}{Y(z)/z-1}w.
				\end{equation}
				By Lemma \ref{lem:injective}, we have
				\begin{equation}
				\frac{Y(z)}{z}-\frac{Y(w)}{w} > \frac{Y(z)}{z}-1>0
				\end{equation}
				when $0<z<1<w<t_0$, and so \eqref{eq:greaterzero} holds for $z\in(0,1)$. Next, by Lemma \ref{lem:injective2}, we have $Y(z)/z\neq Y(w)/w$ when $z\in(-1,0)$ and $1<w<t_0$. Also, $Y(t)/t\to-\infty$ when $t\uparrow0$. By realness and continuity of $Y(t)/t$ in $t\in(-1,0)$ we thus have that
				\begin{equation}
				\frac{Y(z)}{z} < \frac{Y(w)}{w} < 1, \qquad z\in(-1,0),
				\end{equation}
				and so \eqref{eq:greaterzero} also holds for $z\in(-1,0)$.
			\end{proof}
			
			As a consequence of Lemma \ref{lem:nozero}, taking the principal value logarithm in \eqref{logg} when $1<w<1+\epsilon<t_0$, we obtain a function of $z$ that is analytic in the open disk $|z|<t_0$, with branch cut $[1,w]$.
			
			\subsection{Contour integral for (\ref{pk2})}\label{subsect42}

			We next consider the function $z^g-A(z)$ that has its zeros in $|z|\leq 1$ at $z=z_0=1, z_1, \dots, z_{g-1}$, while its other zeros have modulus greater than one. Let $R_0$ be the zero outside $|z|\leq 1$ of smallest modulus; we have $R_0$ is real and larger than one. Take $\epsilon>0$ such that $1+\epsilon < \min\{t_0, R_0\}$ and consider the integral
			\begin{equation}\label{eqn44}
			I(w)=\frac{1}{2\pi i}\oint_{|z|=1+\epsilon}\ln\left(\frac{wY(z)-zY(w)}{Y(z)-z}\right)\,\frac{(z^g-A(z))'}{z^g-A(z)}\,\d z.
			\end{equation}
			\begin{figure}[t]
				\begin{center}
					\begin{tikzpicture}[scale=3]
					\draw[blue, very thick](1,0) circle [radius=0.3];
					\draw[blue, very thick](2,0) circle [radius=0.3];
					\draw[white,fill=white](0.9,-0.02) rectangle (2.2,0.02);
					\draw[blue, thick](1.293,-0.02) -- (1.707,-0.02);
					\draw[blue, thick](1.293,0.02) -- (1.707,0.02);
					\node[above right] at (1,0.25) {$C_1$};
					\node[above right] at (2,0.25) {$C_w$};
					\draw[black](1,0) -- (1.21,0.21);
					\node[right] at (1,0.17) {\small$\delta$};
					\draw[black](2,0) -- (2.21,0.21);
					\node[right] at (2,0.17) {\small$\delta$};
					
					\draw[black](0,0) -- (2.7,0);
					\draw[black](0,0) circle [radius=1];
					\node[below] at (0,0) {$0$};
					\node[below left] at (1.05,0.03) {$1$};
					\draw[black](0,0.03) -- (0,-0.03);
					\draw[black](1,0.03) -- (1,-0.03);
					\draw[black](2,0.03) -- (2,-0.03);
					\draw[black](2.5,0.03) -- (2.5,-0.03);
					\node[below ] at (2.5,-0.0) {$1+\epsilon$};
					\node[below] at (2,0) {$w$};
					\draw[thick,blue,->] (1.51,0.02)--(1.5,0.02);
					\draw[thick,blue,<-] (1.51,-0.02)--(1.5,-0.02);
					\draw[thick,blue,->] (0.7,0.01)--(0.7,0);
					\draw[thick,blue,->] (2.3,-0.01)--(2.3,0);
					\draw[ultra thick,black,dash pattern=on \pgflinewidth off 1.03cm] (0.7,0.35) to [out=130,in=90] (-0.6,0)  ;
					\draw[ultra thick,black,dash pattern=on \pgflinewidth off 1.03cm] (0.7,-0.35) to [out=230,in=-90] (-0.6,0);
					\node[above] at (0.7,0.35) {$z_1$};
					\node[above] at (0.4,0.55) {$z_2$};
					\node[below right] at (0.6,-0.35) {$z_{g-1}$};
					\node[below right] at (0.3,-0.55) {$z_{g-2}$};
					\node[above] at (1.5,0.02) {\small $L_+$};
					\node[below] at (1.5,-0.02) {\small $L_-$};
					\end{tikzpicture}
				\end{center}
				\caption{The four components, $C_1$, $C_w$, $L_+$ and $L_-$, of contour $C$.}
				\label{fig:contours}
			\end{figure}
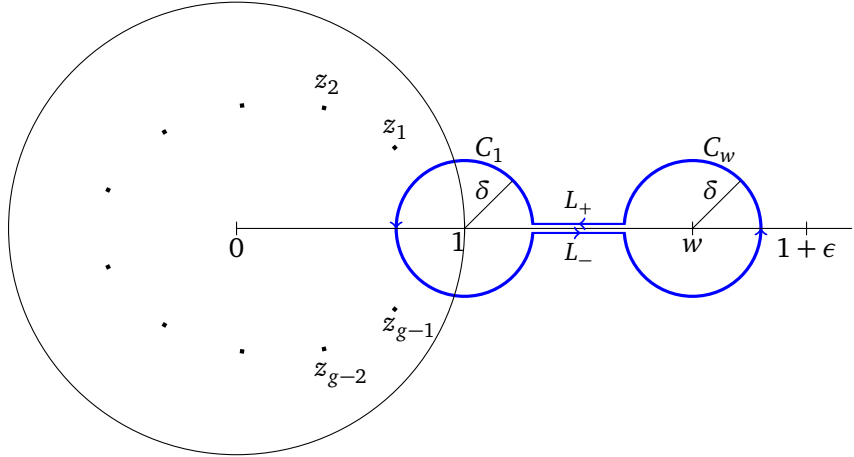

			Choose $\delta>0$ such that $\delta<\frac12 (w-1)$ and $\delta<1+\epsilon-w$ while $|z_k-1|>\delta, k=1,\dots,g-1$. Now let $C$ be the positively oriented contour consisting of the circles $C_1(\delta)$ and $C_w(\delta)$ of radii $\delta$ around 1 and $w$, respectively, together with the line segments $L_{\pm}(\delta)=\{z = t \pm i 0 \,\,|\,\,1+\delta\leq t\leq w-\delta\}$, where $\pm i0:=\lim_{c\downarrow0} \pm ci$. See Figure \ref{fig:contours} for the positioning of the contour $C$ with its four components in the disk $|z|<1+\epsilon$ and relative to the zeros of $z^g-A(z)$. Then, by Cauchy's theorem,
			\begin{align}\label{eqn45}
			I(w)&=\sum_{k=1}^{g-1}\ln\left(\frac{wY(z_k)-z_kY(w)}{Y(z_k)-z_k}\right) + \frac{1}{2\pi i}\oint_{C}\ln\left(\frac{wY(z)-zY(w)}{Y(z)-z}\right)\,\frac{(z^g-A(z))'}{z^g-A(z)}\,\d z.
			\end{align}
			On the line segments $z=t\pm i 0$, $1+\delta\leq t\leq w-\delta$, we use that
			\begin{equation}\label{eqn46}
			w Y(t) - t Y(w) >0 >Y(t)-t.
			\end{equation}
			With the principal value choice for $\ln$, we then get (with $Y(t\pm i 0) = \lim_{c\downarrow0}Y(t \pm i c)$)
			\begin{align}\label{eqn47}
			\ln\left(\frac{wY(t\pm i 0)-(t\pm i 0) Y(w)}{Y(t\pm i 0)-(t\pm i 0)}\right) &= \ln\left(\frac{wY(t)-t Y(w)}{t-Y(t)}\right)\pm \pi i, \quad 1+\delta\leq t\leq w-\delta.
			\end{align}
			Therefore, also using that $t^g - A(t) > 0$, $1<t<1+\epsilon$,
			\begin{align}\label{eqn48}
			&\frac{1}{2\pi i}\oint_{C}\ln\left(\frac{wY(z)-zY(w)}{Y(z)-z}\right)\,\frac{(z^g-A(z))'}{z^g-A(z)}\,\d z\nonumber\\
			=\,&\frac{1}{2\pi i}\int_{1+\delta}^{w-\delta}\left[-\big({\rm ln}\left(\frac{wY(t)-t Y(w)}{t-Y(t)}\right)+\pi i\big)+\big(\ln\left(\frac{wY(t)-t Y(w)}{t-Y(t)}\right)-\pi i\big)\right]\,\frac{(t^g-A(t))'}{t^g-A(t)}\,\d t\nonumber\\
			&+\frac{1}{2\pi i}\oint_{C_1(\delta)} \quad +\frac{1}{2\pi i}\oint_{C_w(\delta)}\nonumber\\
			=\,&-\int_{1+\delta}^{w-\delta}\frac{(t^g-A(t))'}{t^g-A(t)}\,\d t+\frac{1}{2\pi i}\oint_{C_1(\delta)} \quad +\frac{1}{2\pi i}\oint_{C_w(\delta)}.
			\end{align}
			Now, since $g-A'(1)>0$ (due to stability),
			\begin{align}\label{eqn49}
			&\int_{1+\delta}^{w-\delta}\frac{(t^g-A(t))'}{t^g-A(t)}\,\d t = \ln\left(t^g-A(t)\right)\Big|_{1+\delta}^{w-\delta}\nonumber\\
			=\,&\ln\left(w^g-A(w)\right)+O(\delta)-{\rm{ln}}\left[(g-A'(1))\delta+O(\delta^2)\right]\nonumber\\
			=\,&\ln\left(\frac{w^g-A(w)}{g-A'(1)}\right)-\ln\delta+O(\delta),
			\end{align}
			where we have used that
			\begin{equation}\label{eqn410}
			t^g-A(t)=0+(t^g-A(t))'_{t=1}(t-1)+O\big((t-1)^2\big), \quad t\rightarrow 1.
			\end{equation}
			As to the last integral on the last line of \eqref{eqn48}, we use that
			\begin{align}
			w Y(z)-z Y(w)&=(w Y'(w)-Y(w))(z-w)+O(|z-w|^2),\\
			Y(z)-z&=Y(w)-w+O(|z-w|),\\
			z^g-A(z)&=w^g-A(w)+O(|z-w|),
			\end{align}
			with non-vanishing numbers $w Y'(w)-Y(w)$, $Y(w)-w$ and $w^g-A(w)$. Therefore
			\begin{align}\label{eqn414}
			\frac{1}{2\pi i}\oint_{C_w(\delta)}&=\frac{1}{2\pi i}\oint_{C_w(\delta)}\ln\left(\frac{wY(z)-zY(w)}{Y(z)-z}\right)\,\frac{(z^g-A(z))'}{z^g-A(z)}\,\d z\nonumber\\
			&=O(\delta \,\ln \, \delta), \quad \delta\downarrow 0.
			\end{align}
			The middle integral on the last line of \eqref{eqn48} is more delicate since both $Y(z)-z$ and $z^g-A(z)$ vanish at $z=1$. For $z=1+\delta e^{i\phi}$ with $0<\phi<2\pi$ and $\delta\downarrow0$,
			\begin{align}\label{eqn415}
			\frac{wY(z)-z Y(w)}{Y(z)-z} &=\,\frac{w-Y(w)+O(|z-1|)}{1+Y'(1)(z-1)-z+O(|z-1|^2)}\nonumber\\
			&=\,-\frac{w-Y(w)+O(|z-1|)}{(1-Y'(1))(z-1)+O(|z-1|^2)}\nonumber\\
			&=\,-\frac{w-Y(w)}{1-Y'(1)} \,\frac{1}{\delta}e^{-i\phi}\big(1+O(\delta)\big).
			\end{align}
			Hence, since $w-Y(w)>0$, $1-Y'(1)>0$,
			\begin{align}\label{eqn416}
			\ln\left(\frac{wY(z)-z Y(w)}{Y(z)-z}\right) &= \ln\left|\frac{wY(z)-z Y(w)}{Y(z)-z}\right|+i \arg \left(\frac{wY(z)-z Y(w)}{Y(z)-z}\right)\nonumber\\
			&=\ln\left(\frac{w-Y(w)}{1-Y'(1)}\right)-\ln \delta+i(\pi-\phi)+O(\delta).
			\end{align}
			Next, as $z\rightarrow 1$,
			\begin{equation}\label{eqn417}
			\frac{(z^g-A(z))'}{z^g-A(z)}=\frac{g-A'(1)+O(|z-1|)}{(g-A'(1))(z-1)+O(|z-1|^2)}
			=\frac{1}{z-1}+O(1),
			\end{equation}
			since $g-A'(1)>0$. Hence, from \eqref{eqn416} and \eqref{eqn417} with $z=1+\delta e^{i \phi}$ and $dz=i\delta e^{i\phi}d\phi$ in the integral over $C_1$,
			\begin{align}\label{eqn418}
			\frac{1}{2\pi i}\oint_{C_1(\delta)} &= \frac{1}{2\pi i}\int_{0}^{2\pi}\ln\left(\frac{wY(z)-zY(w)}{Y(z)-z}\right)\,\frac{(z^g-A(z))'}{z^g-A(z)}\,i\delta e^{i\phi}\d\phi\nonumber\\
			&= \frac{1}{2\pi i}\int_{0}^{2\pi}\Big[\ln\left(\frac{w-Y(w)}{1-Y'(1)}\right)-\ln \delta+i(\pi-\phi)+O(\delta)\Big]\cdot\Big[\frac1\delta e^{-i\phi}+O(1)\Big] i\delta e^{i\phi}\d\phi\nonumber\\
			&=\ln\left(\frac{w-Y(w)}{1-Y'(1)}\right)-\ln \delta+O(\delta),
			\end{align}
			where we have also used that $\int_0^{2\pi}(\pi-\phi)\,\d\phi=0$.
			
			Using \eqref{eqn49}, \eqref{eqn414} and \eqref{eqn418} in \eqref{eqn48} yields
			\begin{align}\label{eqn419}
			&\frac{1}{2\pi i}\oint_{C}\ln\left(\frac{wY(z)-zY(w)}{Y(z)-z}\right)\,\frac{(z^g-A(z))'}{z^g-A(z)}\,\d z\nonumber\\
			=\,&\ln\left(\frac{g-A'(1)}{w^g-A(w)}\right)+\ln\delta+O(\delta)+O(\delta\ln \delta)+ \ln\left(\frac{w-Y(w)}{1-Y'(1)}\right)-\ln \delta+O(\delta)\nonumber\\
			=\,&\ln\left(\frac{g-A'(1)}{w^g-A(w)} \cdot \frac{w-Y(w)}{1-Y'(1)}\right)+O(\delta).
			\end{align}
			Returning then to \eqref{eqn44}-\eqref{eqn45}, letting $\delta\downarrow0$, we see that
			\begin{equation}\label{eqn420}
			I(w)=\ln\left[\frac{g-A'(1)}{w^g-A(w)}\cdot \frac{w-Y(w)}{1-Y'(1)}\prod_{k=1}^{g-1}\frac{wY(z_k)-z_kY(w)}{Y(z_k)-z_k}\right]=\ln\big[X_g(w)\big]
			\end{equation}
			by \eqref{pgf4}. Here we have also used that the zeros $z_k$ are real or come in conjugate pairs so that for $w\in(1,1+\epsilon)$ by \eqref{eq:greaterzero} in Lemma \ref{lem:nozero} both $X_g(w)$ and the product $\prod_{k=1}^{g-1}$ in \eqref{eqn420} are real and positive, with
			\begin{equation}\label{eqn421}
			\ln\left(\prod_{k=1}^{g-1}\frac{wY(z_k)-z_kY(w)}{Y(z_k)-z_k}\right)=\sum_{k=1}^{g-1}\ln\left(\frac{wY(z_k)-z_kY(w)}{Y(z_k)-z_k}\right).
			\end{equation}
			This proves \eqref{pk2} for $w\in(1,1+\epsilon)$.

			\subsection{Completion of the proof}
			
			The extension of the validity range of \eqref{pk2} beyond the set $1<w<1+\epsilon$ is compromised by the appearance of the factor $\ln[(w Y(z)-z Y(w))/(Y(z)-z)]$ in the integrand.
			The validity range can be extended to an open set containing the interval $[0,1]$, allowing computation of moments and derivatives. To see this, let
			\begin{equation}\label{eqn431}
			Q(z,w)=\frac{w Y(z) - z Y(w)}{Y(z)-z} =Y(w)\frac{1-\frac{Y(z)/z}{Y(w)/w}}{1-Y(z)/z}, \qquad |z|,|w| \leq 1+\epsilon.
			\end{equation}
			For $0\leq w\leq 1$ and $|z|=1+\epsilon$,
			\begin{equation}\label{eqn432a}
			0<Y(0)\leq Y(w)\leq 1, \quad \left|\frac{Y(z)/z}{Y(w)/w}\right|\leq \left|\frac{Y(z)}{z}\right|\leq \frac{Y(1+\epsilon)}{1+\epsilon}<1,
			\end{equation}
			and so $Q(z,w)$ is bounded away from $(-\infty,0]$ when $0\leq w\leq 1$ and $|z|=1+\epsilon$. By continuity of $Q$ as a function of $w$, this continues to hold for $w$ in an open set $\Omega$ containing $[0,1]$ and $|z|=1+\epsilon$. This implies that $\ln Q(z,w)$ is analytic in $w\in\Omega$, with principal value $\ln$, extending the validity of \eqref{pk2} to $w\in\Omega$ by analyticity. We have extensive numerical evidence that the set of $w$ for which $Q(z,w) \not\in (-\infty,0]$, all $z$ with $|z|=1+\epsilon$, contains a disk around 0 with radius not significantly smaller than $1+\epsilon$. This would extend the validity of \eqref{pk2} beyond the unit disk $|w|\leq 1$.

			We now re-express the integral form in \eqref{pk2} to a form that is valid for all $w$, $|w|<1+\epsilon$. We choose here $\epsilon$ such that $1+\epsilon<\min\{t_0,R_0\}=:1+\epsilon_0$ as in Subsection \ref{subsect42}. Let $w$ be fixed with $1<w<1+\epsilon$. We compute for $|z|=1+\epsilon$
			\begin{equation}\label{eqn432}
			\frac{(z^g-A(z))'}{z^g-A(z)}=\frac{g}{z}+\frac{\big(1-\frac{A(z)}{z^g}\big)'}{1-\frac{A(z)}{z^g}}=\frac{g}{z}+\frac{\d}{\d z}\left[\ln\left(1-\frac{A(z)}{z^g}\right)\right],
			\end{equation}
			where we can choose the principal value of $\ln$ since
			\begin{equation}\label{eqn433}
			\left|\frac{A(z)}{z^g}\right| \leq \frac{A(1+\epsilon)}{(1+\epsilon)^g}<1, \quad |z|=1+\epsilon.
			\end{equation}
			As in \eqref{eqn47}-\eqref{eqn48},
			\begin{align}\label{eqn434}
			&\frac{1}{2\pi i}\oint_{|z|=1+\epsilon}\ln\left(\frac{w Y(z)-z Y(w)}{Y(z)-z}\right)\frac{g}{z}\d z\nonumber\\
			=\,& g \left.\ln\left(\frac{w Y(z)-z Y(w)}{Y(z)-z}\right)\right|_{z=0}+\frac{g}{2\pi i} \oint_C \ln\left(\frac{w Y(z)-z Y(w)}{Y(z)-z}\right)\frac{\d z}{z}\nonumber\\
			=\,& g \ln w-g \int_{1+\delta}^{w-\delta} \frac{\d z}{z}+O(\delta \ln \delta)\nonumber\\
			=\,& g \ln w-g \ln\left(\frac{w-\delta}{1+\delta}\right)+O(\delta \ln \delta),
			\end{align}
			and this vanishes as $\delta\downarrow0$. Therefore, see \eqref{eqn44},
			\begin{align}\label{eqn435}
			I(w) &= \frac{1}{2\pi i}\oint_{|z|=1+\epsilon}\ln\left(\frac{w Y(z)-z Y(w)}{Y(z)-z}\right)\frac{\d}{\d z}\left[\ln\left(1-\frac{A(z)}{z^g}\right)\right]     \d z\nonumber\\
			&=\frac{-1}{2\pi i}\oint_{|z|=1+\epsilon}\frac{\d}{\d z}\left[\ln\left(\frac{w Y(z)-z Y(w)}{Y(z)-z}\right)\right]\ln\left(1-\frac{A(z)}{z^g}\right)\d z,
			\end{align}
			where we have used partial integration with continuous differentiable functions $\ln\left(1-A(z)/z^g\right)$ and $\ln\left[(w Y(z)-z Y(w))/(Y(z)-z)\right]$ on the closed contour $|z|=1+\epsilon$. We compute
			\begin{equation}\label{eqn436}
			\frac{\d}{\d z}\left[\ln\left(\frac{w Y(z)-z Y(w)}{Y(z)-z}\right)\right]=\frac{Y'(z)z-Y(z)}{Y(z)-z} \,\frac{Y(w)-w}{w Y(z)-z Y(w)},
			\end{equation}
			and obtain
			\begin{equation}\label{eqn437}
			I(w) = \frac{-1}{2\pi i}\oint_{|z|=1+\epsilon}\frac{Y'(z)z-Y(z)}{Y(z)-z} \,\frac{Y(w)-w}{w Y(z)-z Y(w)}\ln\left(1-\frac{A(z)}{z^g}\right) \d z,
			\end{equation}
			which is valid for any $w\in(1,1+\epsilon)$.
			
			We now extend \eqref{eqn437} to all $w$ with $|w|<1+\epsilon$ by using Lemma \ref{lem:nozero}. Let $0<\epsilon_1<\epsilon$. We have $|Y(z)-z|>0$ when $|z|=1+\epsilon$ and
			\begin{equation}\label{eqn438}
			\big|w Y(z)-z Y(w)\big|>0 \quad \text{ when }|z|=1+\epsilon, |w|\leq1+\epsilon_1,
			\end{equation}
			by Lemma \ref{lem:nozero} and $Y(0)\neq 0$. Therefore, by continuity and compactness,  $(w Y(z)-z Y(w))(Y(z)-z)$ is bounded away from 0 when $|z|=1+\epsilon$ and $|w|\leq 1+\epsilon_1$. This implies that the right-hand side of \eqref{eqn437} is analytic in $w$, $|w|<1+\epsilon_1$, by analyticity of $Y$. Since $X_g(w)=\exp(I(w))$ for $1<w<1+\epsilon$, we then get by analyticity of $X_g$ that
			\begin{equation}\label{eqn439}
			X_g(w)=\exp\left(\frac{-1}{2\pi i}\oint_{|z|=1+\epsilon}\frac{Y'(z)z-Y(z)}{Y(z)-z}\,\frac{Y(w)-w}{w Y(z) - z Y(w)}
			\ln\left(1-\frac{A(z)}{z^g}\right)\,\d z\right)
			\end{equation}
			holds for all $w$, $|w|\leq 1+\epsilon_1$ and any $\epsilon_1\in(0,\epsilon)$. Then a simple rearrangement of the integrand in \eqref{eqn439} yields Theorem \ref{mainttt}.
			
			\section{Conclusions}\label{sec:con}
			We have presented novel formal solutions for the FCTL queue in the form of contour integrals. Theorem \ref{mainttt} presents the contour-integral representation for the PGF of the overflow queue. From this PGF, essentially all relevant information about the stationary behavior of the FCTL queue can be obtained, by taking derivatives at one for the moments, derivatives at zero for the distribution, and by using simple recursions to obtain the queue lengths at all moments within the cycle and the stationary delay distribution. A contour integral expression for the first moment was obtained recently by Oblakova et al.~\cite{ahmad}, and the present paper can be seen as an extension of that work. The two papers together present an alternative approach for the FCTL queue and its generalizations, using contour integrals instead of factorizations in terms of complex roots that need to be determined numerically.
			
			A possible thread for future research relates to asymptotics.
			In classical queueing theory, a prominent line of research is related to heavy traffic, an asymptotic regime in which the traffic intensity approaches 100\%. Next to more probabilistic methods such as weak convergence techniques and coupling, another way to obtain heavy-traffic results is through the asymptotic evaluation of Pollaczek-type integrals; see e.g.~\cite{cohen,kingman1962queues} for single-server queues and \cite{britt1} for classical bulk-service queues. Now that Pollaczek-type integrals for the FCTL queue are available, it is worthwhile to explore the possibilities for heavy-traffic analysis.

			\section*{Acknowledgments}
			
			We thank the authors of \cite{ahmad} for sharing an early version of their work. This work is supported by the NWO Gravitation Networks grant 024.002.003.
			The work of JvL is further supported by an NWO TOP-GO grant and by an ERC Starting Grant. The research of RT is done as part of the {\sc Dynafloat} project, funded by NWO, grant number 438-13-206.

			\bibliographystyle{abbrv}
			

			\appendix
			\section{Root-finding algorithm}\label{app:alg}
			We now present a root-finding algorithm and some supporting results. A similar algorithm was used in \cite{BvLfctlcorrelated}.
			The idea behind the algorithm is that roots of polynomial equations are generally easy to find, at least numerically. Therefore, we approximate $A(z)$ (which typically is a non-polynomial function) with its Taylor series $A_n(z)$ of order $n$. Solving this truncated equation boils down to root-finding of a polynomial. If the roots of the truncated equation are sufficiently close to the roots of $z^g-A(z)$, we can find the latter roots easily from the roots of $z^g-A_n(z)$ by using a Newton-Raphson type method. 
			
			\begin{algorithm}[H]
				\caption{Root-finding based on truncated Taylor series of $z^g-A(z)$.} \label{alg:root}
				\begin{algorithmic}[1]
					\State Input: $A(z),g$ (and often $A(z)=Y(z)^c$).
					\State Define: $D(z)=z^g-A(z)$.
					\State Compute $n$: $\max\{100,50+\max\{c,g\}\}$.
					\State Compute Taylor expansion $D_n(z)$ of $D(z)$ of order $n$.
					\State Numerically solve $D_n(z)=0$ for $|z|\leq1$, obtaining roots $\hat{z}_1,...,\hat{z}_g$.
					\State Use $\hat{z}_1,...,\hat{z}_g$ as input for a method to find the roots of $D(z)$ for $|z|\leq1$, obtaining roots $z_1,...,z_g$.
					\State Return $z_1,...,z_g$.
				\end{algorithmic}
			\end{algorithm}
			
			
			\vspace{0.3 cm}
			\noindent We now present two propositions, in support of Algorithm \ref{alg:root}. The propositions state that under very mild conditions the number of roots of $z^g-A(z)$ on or within the unit circle are equal to the number of roots of the truncated equation $z^g-A(z)$ on or within the unit circle and that the roots of the truncated equation converge to the roots of $z^g-A(z)$ (when $n$ tends to infinity).
			
			\begin{proposition}\label{prop:roots}
				Let $D(z)=z^g-A(z)$ and let $D_n(z) := z^{g}-A_n(z)$, where $A_n(z)$ denotes the $n$-th order Taylor approximation of $A(z)$. Upon assuming that $A(z)$ is a PGF; that $A(z)$ is analytic in the disk $|z|<1+\delta$ for some $\delta>0$; and that $g< A^\prime(1)$, $D_n(z)=0$ has as many roots on or within the unit circle as $D(z)$ (i.e. $g$).
			\end{proposition}
			
			\begin{proof}
				Rouch\'{e}'s theorem says that if $f$ and $g$ are analytic inside some region $K$ with closed contour $\partial K$ and if $|g(z)|<|f(z)|$ on $\partial K$, then $f$ and $f+g$ have the same number of zeros inside $K$.
				
				The conditions that $A(z)$ has to be analytic in $|z|<1+\delta$ and $g< A^\prime(1)$ together imply
				\begin{equation} \label{eq:ineq}
				(1+\gamma)^g >A(1+\gamma),
				\end{equation}
				for some $\gamma\in(0,\delta)$, see e.g. \cite{darroch}. Assume $|z|=1+\gamma$. Then:
				\begin{align}
				|z|^g & = (1+\gamma)^g
				> A(1+\gamma)
				\geq A_n(1+\gamma)
				= A_n(|z|)
				\geq |A_n(z)|,
				\end{align}
				where the strict inequality follows from \eqref{eq:ineq} and the remaining inequalities from the fact that $A(z)$ is a PGF.
				So we may apply Rouch\'{e}'s theorem on $f(z)=z^g$ and $g(z)=-A_n(z)$. As $z^g$ has $g$ roots on or within the unit circle, $D_n(z)$ will have $g$ roots as well, just as $D(z)$.
			\end{proof}
			
			\begin{proposition}\label{prop:dist}
				Let $D(z)$ and $D_n(z)$ be as defined in Proposition \ref{prop:roots}. Let $z_j$, $j=1,...,g$ be the roots of $D(z)$ on or within the unit circle. Then
				\begin{equation}
				|D_n(z_j)| \leq \sum_{j=n+1}^\infty a_k,
				\end{equation}
				for $j=1,...,g$, where $a_k = \P(A=k)$.
			\end{proposition}
			\begin{proof}
				We directly obtain from the definition of $z_j$
				\begin{align}
				|D_n(z_j)| & = |D_n(z_j)-D(z_j)| \nonumber\\
				& =\left|z_{j}^g-\sum_{i=0}^n a_iz_{j}^i-z_{j}^g+\sum_{i=0}^\infty a_iz_{j}^i \right|
				= \left|\sum_{i=n+1}^\infty a_iz_{j}^i\right|
				\leq \sum_{i=n+1}^\infty a_i,
				\end{align}
				because $a_i\geq0$ and $|z_j|\leq 1$.
			\end{proof}
			From Proposition \ref{prop:dist} we see that if we let $n$ tend to infinity, then $D_n(z_j)$ tends to 0. This implies that the roots obtained by using $D_n(z)$ will be close to the actual roots of $D(z)$ when $n$ is sufficiently high.
			
			\vspace{.3cm} \noindent {\bf Roots for the Poisson case}
			Take $A(z)=\e^{c\lambda(z-1)}$ and let $W(\cdot)$ denote the Lambert W-function.
			Then
			\begin{equation}\label{eq:LambertW}
			z_k=-\frac{g}{c\lambda}W\left(-\frac{c\lambda}{g}e^{2\pi i k/g}\e^{-c\lambda/g}\right), \quad k=1,\ldots,g-1
			\end{equation}
			with $W(\cdot)$  the principal value of the Lambert W-function and $i$ is the imaginary unit satisfying $i^2=-1$.
		\end{document}